\documentclass{birkjour}
\usepackage{amscd,amssymb,hyperref,mathtools,url}
\usepackage[
backend=biber,
style=ieee,
sorting=ynt
]{biblatex}
\usepackage{datetime}
\usepackage[all]{xy}
\usepackage{graphicx}
\usepackage{enumerate}
\usepackage{verbatim}
\usepackage{wrapfig}
\usepackage{siunitx}
\usepackage{biblatex}
\newtheorem{theorem}{Theorem}[section]
\newtheorem{corollary}[theorem]{Corollary}
\theoremstyle{definition}
\newtheorem{definition}[theorem]{Definition}

\newtheorem{lemma}[theorem]{Lemma}
\newtheorem*{remark}{Remark}

\newtheorem*{example}{Example}
\numberwithin{equation}{section}

\newcommand{\R}{\mathbb{R}}

\newcommand{\be}{\begin{enumerate}}
\newcommand{\ee}{\end{enumerate}}
\newcommand{\bi}{\begin{itemize}}
\newcommand{\ei}{\end{itemize}}

\newcommand{\ba}{\begin{array}}
\newcommand{\ea}{\end{array}}
\newcommand{\bmat}{\left[\begin{array}}
\newcommand{\emat}{\end{array}\right]}
\newcommand{\bt}{\begin{thm}}
\newcommand{\et}{\end{thm}}
\newcommand{\bp}{\begin{proof}}
\newcommand{\ep}{\end{proof}}
\newcommand{\bprop}{\begin{prop}}
\newcommand{\eprop}{\end{prop}}
\newcommand{\bl}{\begin{lemma}}
\newcommand{\el}{\end{lemma}}
\newcommand{\bc}{\begin{cor}}
\newcommand{\ec}{\end{cor}}
\newcommand{\bd}{\begin{defn}}
\newcommand{\ed}{\end{defn}}

\newcommand{\vep}{\varepsilon}

\newcommand{\tx}{\textrm}

\begin{document}

\title[Topological approaches to knotted electric charge distributions]{Topological approaches to knotted electric charge distributions}
\author[Max Lipton]{Max Lipton}
\address{Department of Mathematics\br
Cornell University\br
Ithaca NY, 14853\br
USA}
\email{ml2437@cornell.edu}

\thanks{This article was partially funded by a NSF Research Training Group grant titled Dynamics, Probability, and PDEs in Pure and Applied Mathematics, DMS-1645643.}


\date{\today}

\begin{abstract}
Consider a knot $K$ in $S^3$ with uniformly distributed electric charge. Whilst solutions to the Laplace equation in terms of Dirichlet integrals are readily available, it is still of theoretical and physical interest to understand the qualitative behavior of the potential, particularly with respect to critical points and equipotential surfaces. In this paper, we demonstrate how techniques from geometric topology can yield novel insights from the perspective of electrostatics. Specifically, we show that when the knot is sufficiently close to a planar projection, we prove a lower bound on the size of the critical set based on the projection's crossings, improving a 2019 result of the author. We then classify the equipotential surfaces of a charged knot distribution by tracking how the topology of the knot complement restricts the Morse surgeries associated to the critical points of the potential.
\end{abstract}

%%% ----------------------------------------------------------------------
\maketitle
%%% ----------------------------------------------------------------------
\section{Introduction}
Imagine a closed knotted wire in space fixed in place, with uniform electric charge distributed on it.  Our novel problem of interest is to understand the properties of the knot's electric field, a topic with applications to the study of knotted DNA molecules in viruses \cite{orl17}, the design of electrical circuits \cite{gp05}, and broader questions in physical knot theory \cite{kauf91}. In our work, the knot is fixed in space. By contrast, others have studied how knotted curves deform, as discussed in the extensive literature on the energy methods and flows of knots \cite{buck95, freedman94}.

The electric field around a charged knot is the negative gradient of a potential function, so the fixed points of the electric field correspond to critical points of the potential. The primary object of study we will focus on are the critical points of a knot which are structurally stable under small perturbations of the charge distribution. Physically, these are the points where all of the electric forces cancel. Thus some basic questions arise: How many critical points are there? What are their local behaviors? And how does the critical set relate to topological properties of the knot? 

In this paper, we present some results which address these broad questions through the use of techniques from geometric topology, and offer some rigorous proofs for some of the results communicated in \cite{lst22}. Whilst the methods we apply are standard in low-dimensional topology, research in electrostatics and and physical knot theory rarely implement these methods.  We use a topological counting argument to show that the size of a knot's structurally stable critical set sharpens from a previously proven lower bound based on the knot's tunnel number \cite{lipton}. Next, we calculate the effect of the Morse surgery on their genera. A critical point of Morse index $1$ corresponds to an increase in genus as we raise the level of the potential, whilst a critical point of index $2$ corresponds to a decrease in genus. The tuple of equpotential surface genera is a Scharlemann-Thompson handlebody width decomposition of the knot complement \cite{st94}, but our tuple differs from the the classical thin decomposition in that we exclusively consider those arising from potentials, and we do not lexicographically reorder the tuples.

Let $K \subseteq \R^3 \subseteq S^3$ be a smooth knot parametrized by the curve $r(t),$ $t \in [0,2\pi]$ with $r(0) = r(2\pi)$. We will take the convention that $S^3$ is the union of $\R^3$ and a single compactifying point at infinity. Suppose $K$ is endowed with a uniform charge distribution. With a choice of units, the electric potential between a point $k \in K$ and a point charge at $x$ at a distance $R$ from $k$ is proportional to $R^{-1}$ by Coulomb's Law. It therefore makes sense to define the electric potential $\Phi: S^3 - K \to \R$, on the complement of $K$ by the line integral 

\noindent
\begin{equation}
\Phi(x) = \int_{k \in K} \frac{dk}{|x - k|} = \int_0^{2\pi} \frac{|r'(t)|}{|x - r(t)|}dt, \tx{ } x \in \R^3 - K.
\label{this}
\end{equation}

We set $\Phi(\infty) = 0$ to ensure smoothness. By differentiating under the integral sign with respect to $x$, we can see the electric potential is smooth and harmonic. The electric field is defined by $E = - \nabla \Phi$. The critical points of $\Phi$ represent equilibrium points where a charged test particle at rest will continue to experience no electric force from the charge distribution. Some conventions define the potential to be negative, in order for the electric field to point towards the knot, but it is more convenient for our purposes to work with a nonnegative potential function. Typically charged wires are presented in physics literature as the integral over a solid, thin torus, where our idealized uniform charge distribution over a curve is approximated by a nonuniform charge distribution over a solid, which depends greatly on the geometry of the parametrization. However, our work is concerned with critical points which occur far away from the knot, so this distinction does not alter our results.

The electric potential is difficult to analyze in general, as integral formulas can be highly nonlinear and nonsymmetric, but the potential and electric field integrals can be numerically approximated via Gaussian quadrature, and the solutions to $E(x) = 0$ can be determined by the multivariable Newton method. Initial numerical approximations of the critical sets for various knots were implemented by Townsend and Lipton \cite{knotcode}, showing that there are isotopic parametrizations of the unknots which have critical sets of differing size. In particular, the trivial unknot embedding $r(t) = (\cos(t),\sin(t),0)$ has a single critical point in the origin, whilst the $(5,1)$ torus knot with parametrization $r(t) = (\cos(t + 2)\cos(5t), \cos(t+2)\sin(5t),-\sin(t))$ has three critical points, as depicted in Figure 1.

\begin{figure}[h]
\centerline{\includegraphics[width=2in]{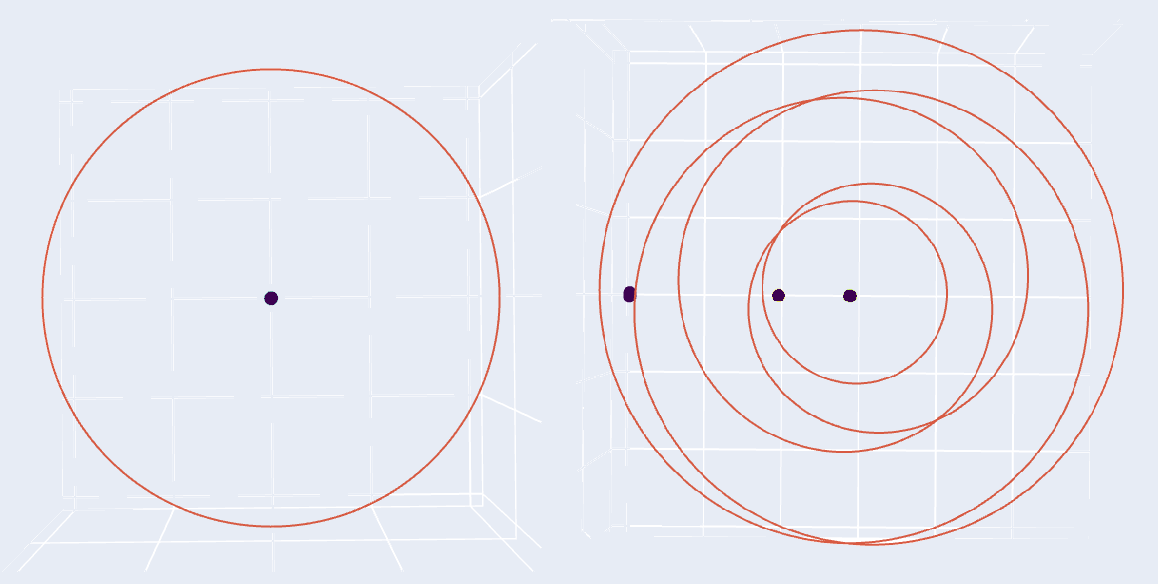}}
\caption{A top-down projection of two knots isotopic to the unknot with differing critical sets.}
\label{unknotcritset}
\end{figure}

It is unsurprising that the critical set is not invariant under isotopy, given the dependence of $\Phi$ on $r$, but a 2019 result \cite{lipton} proved there is indeed a restriction on the critical set based on a knot invariant called the tunnel number. Let $cp(K)$ denote the smallest size of structurally stable critical sets of $K$ over all parametrizations, which includes the critical point at infinity in the count. Let $t(K)$ denote the tunnel number, which is the smallest number of arcs one must add to $K$ such that the complement is a handlebody. The aforementioned result is stated as follows.
\begin{theorem}
\label{tunnel}
For all knots $K$, $cp(K) \geq 2t(K) + 2$.
\end{theorem}

Theorem \ref{tunnel} was proven using techniques in geometric topology, Morse theory, and stable manifold theory. The proof does not apply Morse reconstruction to $\Phi$ itself, but rather, it applies the Morse Rearrangement Lemma (see Lemma 2.4.12 of \cite{nico}) to $\Phi$, which allows us to alter the global, long-term dynamical behavior of the gradient to conform to certain convenient properties, whilst leaving the local behavior around the critical points unchanged. 

Critical sets of harmonic functions can be pathological in general, but the conditions we impose avoid these edge cases, and so the critical sets under consideration will always be finite \cite{sh80} and lie in the convex hull of the charge distribution \cite{walsh48}, which we assume is bounded. The inclusion of the structurally stable requirement in our definition of $cp(K)$ is important, because as we will discuss in Section 2, generic knot isotopies (under the $C^\infty$ topolgy) can induce bifurcations which create or annihilate pairs of critical points, with a structurally unstable degenerate critical point of high multiplicity occuring during the transition. Demanding that the potential is Morse avoids undercounting the critical points. 

In Section 3, we prove a partial sharpening of Theorem \ref{tunnel}. The tunnel number is a poor measure of a knot's complexity. For instance, all torus knots have tunnel number $1$, even though they can have arbitrarily high crossing number. Likewise, prime knots with up to $10$ crossings can only have tunnel numbers $1$ or $2$ \cite{mori96}. However, we can show the bound sharpens when the knot is sufficiently close to a planar curve, defined by the notion of a $\delta$-lifting which is given later on. We then show that given a planar knot diagram with $c$ crossings, there is a sufficiently small $\delta > 0$ such that the size of the critical set for its $\delta$-liftings is at least $2c+2$. This result is proven by applying the Poincare-Hopf Index Theorem to the electric field of the projected knot to prove the existence of a critical point in each connected component in the projection's planar complement, of which there are $c+2$, including the critical point at infinity which corresponds to the unbounded component. We then use the Morse inequalities to show that lifting each crossing from the projection necessitates the creation of another critical point due to the change in the topology of the knot complement, which yields our final count of $2c+2$.

In Section 4, we prove a result describing the topology of the equipotential surfaces in terms of the critical set. The critical values of $\Phi$ partition the positive real numbers into finitely many intervals, and a key idea from Morse theory states that the level sets of two values from the same interval are homeomorphic. This leads us to the central object of the section: the Morse code of a knot, which is a tuple of numbers listing the genera of the equipotential surfaces corresponding to each of the intervals. Another key idea from Morse theory states the passage from one equipotential surface to another is determined, up to homotopy equivalence, by the attachment of a disc of appropriate dimension. Even though the attachment maps can be quite pathological in general, we prove that the potential yields particularly simple surgeries, showing that a critical point of index $1$ corresponds to an increase in the subsequent term of the Morse code, whilst a critical point of index $2$ corresponds to a decrease.

\section{Bifurcations through degenerate critical points}

Recall that for a a manifold $M$, a Morse function $f: M \to \R$ is a smooth, real-valued function whose critical points are nondegenerate, which means the Hessians at the critical points are nonsingular. Throughout this work, we assume the electric potential is Morse, which is a generic property in the space of real-valued functions on the knot complement \cite{milnor,nico}. A knot isotopy induces a path in this space via the corresponding potentials, but the path could pass through non-Morse functions. The proofs of our theorems make liberal use of the Morse property, and break down when $\Phi$ is not Morse.

However, using Cerf theory, we can extract a generic property of knot isotopies in which the path through non-Morse potentials occurs only finitely many times, and their degenerate critical points can be classified based on partial derivatives of the isotopy. We will see that these degenerate critical points correspond to saddle-node bifurcations. It should be emphasized that this result does not describe the dynamics of a moving charged knot, as we do not take into account the induced magnetic forces from the moving charges. Rather, we are describing a bifurcation among a smooth family of possible fixed, rigid knot configurations.

We now state a part of the Cerf Structure Theorem in terms of the homotopies of Morse functions. There is a more general structure theorem given by a stratification \cite{nico,freed} of the class of isotopies via the coordinate expansion of general degenerate critical points. A full exposition can be found in Ch. 23 of Freed \cite{freed}.

\begin{theorem}
\label{cerf}
Let $\tilde{F}: M \times [0,1] \to \R$ be a one parameter family of smooth functions on an $n-$manifold $M$ indexed by the variable $s$. We can replace $\tilde{F}$ with a substitute function (which we still call $\tilde{F})$ in a dense subspace of $C^\infty(M \times [0,1])$ which satisfies the following. Each $F_s = \tilde{F}(-, s)$ is Morse except for finitely many $s$ where there could exist degenerate critical points. For each such $s_0$, and for each degenerate critical point $p$ of $F_{s_0}$, there exist local coordinates $(x_1, \dots, x_n, s)$ of $M \times [0,1]$ centered at $(p,s_0)$ such that $\tilde{F}(x_1, \dots, x_n,s) = x_1^3 + \vep_1 sx_1 + \vep_2 x_2^2 + \dots + \vep_n x_n^2 + C$, where each $\vep_i \in \{\pm 1\}$, whose values depend on the partial derivatives of $\tilde{F}$, and $C$ is the degenerate critical value.
\end{theorem}

We will apply this theorem to the one parameter family of electric potentials induced by a knot isotopy. Suppose $K_0$ and $K_1$ are two smooth isotopic knots where the isotopy is parametrized by $r:[0,2\pi] \times [0,1] \to \R^3$. For fixed $s_0 \in [0,1]$, we will say $r_{s_0} = r(-,s_0)$ is the parametrization of the knot $K_{s_0}$, whose electric potential is $\Phi_{s_0}$. We will now apply Theorem \ref{cerf} to the one parameter family $\Phi_s$, and if necessary, replace $\Phi_s$ with another nearby generic family so that the property mentioned in the theorem holds.

We will label the coordinates from Theorem \ref{cerf} as $(x,y,z,s)$ so we have that the local coordinate formula of $\Phi$ at a specific $s_0$ where $\Phi_{s_0}$ has a degenerate critical point is $\Phi(x,y,z,s) = x^3 + \vep_1sx + \vep_2y^2 + \vep_3z^2 + C$. The convention for a knot isotopy is  to use a globally defined $s$ ranging from $0$ to $1$, but in our current coordinate chart, the local $s$-coordinate is centered at zero, and can be positive or negative. 

We can classify degenerate critical points of the electric potential as bifurcations according to the signs of $\vep_1,\vep_2$, and $\vep_3$. In these coordinates, the gradient and Hessian (with respect to the space coordinates alone) are 
\begin{align}
\nabla\Phi_s (x,y,z) &= (3x^2 + \vep_1sx, 2\vep_2y, 2\vep_3z) \\
H\Phi_s (x,y,z) &= \begin{bmatrix} 6x & 0 & 0 \\ 0 & 2\vep_2 & 0 \\ 0 & 0 & 2\vep_3
\end{bmatrix}.
\end{align}
For fixed $s \neq 0$, $\Phi_s$ must be Morse, and within this chart, $\Phi_s$ has critical points at $x = \pm \sqrt{\frac{-\vep_1s}{3}}, y = 0, z = 0$, which means there are either two critical points when $\vep_1$ and $s$ differ in sign, or no critical points when $\vep_1$ and $s$ are of the same sign. Furthermore, for a fixed $s$ where there are two critical points, they are of differing indices of $1$ and $2$, as the two $x$-values of the critical points differ in sign. We call these two critical points a critical pair.

Not every possible configuration of signs for the $\vep_i$ will result from a potential homotopy induced from a knot isotopy, as each $\Phi_s$ is harmonic and hence cannot contain local maxima or minima. For instance, when $\vep_2 = \vep_3 = -1$, the Hessian formula implies the existence of a critical point of index $3$ when $\tx{sgn}(s) = -\tx{sgn}(\vep_1)$, which cannot exist. This classification is summarized in Figure \ref{table}.

\begin{figure}
 \begin{tabular}{||c c c c||} 
 \hline
 $\vep_1$ & $\vep_2$ & $\vep_3$ & Bifurcation for increasing s\\ 
 \hline\hline
 $1$ & $1$ & $1$ & Impossible \\
 \hline
 $1$ & $1$ & $-1$ & Destruction of critical pair \\
 \hline
  $1$ & $-1$ & $1$ & Destruction of critical pair \\
 \hline
  $1$ & $-1$ & $-1$ & Impossible \\
 \hline
  $-1$ & $1$ & $1$ & Impossible \\
 \hline
  $-1$ & $1$ & $-1$ & Creation of critical pair \\
 \hline
  $-1$ & $-1$ & $1$ & Creation of critical pair \\
 \hline
  $-1$ & $-1$ & $-1$ & Impossible \\
 \hline
\end{tabular}
\caption{The generic classification of bifurcations of critical points of the potential induced by knot isotopy.}
\label{table}
\end{figure}

\begin{remark}
Symmetries in the parametrization and isotopy of $K$ can result in non-saddle-node bifurcations, as a pitchfork bifurcation occurs in the numerical implementation of the ``flattening" isotopy of a Figure 8 knot, as shown in joint work by the author, Townsend and Strogatz \cite{lst22}. We considered the critical set of the Figure-8 knot parametrized by $r(t) = ((2 + \cos{2t})\cos{3t}, (2 + \cos{2t})\sin{3t},\sin{4t})$, and observed the bifurcation as the knot ``flattens" by scaling the third coordinate function by a factor of $s$, as $s$ varies from $1$ to $0$. At approximately $s = 0.41$, we saw the annihilation of a pair of critical points at the location of a third critical point which persists. Nevertheless, in order to preserve the topological restrictions on the critical set, as explained in the next section, changes in the critical set must still take the form of creation and destruction of pairs.

\end{remark}

\section{A partial sharpening of the tunnel number bound}

Recall that the index of a critical point is the number of negative eigenvalues of its corresponding Hessian, which is invariant under a change of local coordinates. If $m_i$ denotes the number of critical points of index $i$, then the Morse inequalities imply the alternating sum 
\begin{equation}
\sum\limits_{i=0}^{\tx{dim} M} (-1)^i m_i = \chi(M),
\end{equation}
with $\chi$ denoting the Euler characteristic \cite{nico}.

Now for a given knot $K$, with electric potential $\Phi$, we will let $m_i$ denote the number of critical points of $\Phi$ with index $i$. As $\Phi$ is harmonic, we have that there are no local extrema save for the point at infinity, which is a global minimum. As the Euler characteristic of any knot complement is $0$, we get the following lemma.

\begin{lemma}
\label{m2m1}
Regardless of the knot type of $K$, $m_2 = m_1 - 1$.
\end{lemma}

We now define the notion of a $\delta$-lifting of a planar curve. The electric field of a uniformly charged planar curve is straightforward to describe, and if $K$ is a $\delta$-lifting of a planar curve for sufficiently small $\delta$, then the critical points of $K$ will include those of the projection, in addition to new critical points obtained at each crossing.

\begin{definition} Let $\delta > 0$ and $\gamma$ be a closed curve in the plane $P \subseteq \R^3$. We say the knot $K$ is a \textbf{$\delta$-lifting of $\gamma$} if we can obtain $K$ by taking a ball $B$ of radius $\delta$ centered at each crossing of $\gamma$ and replacing $\gamma \cap B$ with two disjoint arcs contained in $B$ whose end points are in $\gamma \cap \partial B$.
\end{definition}

We now come to the main theorem of this section.

\begin{theorem}
\label{knotflattening}
Let $\tilde{K}$ be a knot type, and suppose $\gamma$ is a planar projection of some tame parametrization of type $\tilde{K}$ with $c$ crossings. Then there exists a $\delta > 0$ such that if $K$ is a $\delta$-lifting of $\gamma$ with knot type $\tilde{K}$, then the electric potential of $K$ (after a possible perturbation) has at least $2c + 2$ critical points, including the critical point at infinity.
\end{theorem}

\begin{remark}
 As $c \geq t(K)$, this bound is a sharpening of Theorem \ref{tunnel}.
\end{remark}

The proof of Theorem \label{knotflattening} requires a short, elementary lemma.

\begin{lemma}
\label{eulerknot}
There are $c+1$ bounded connected components of the planar complement of $\gamma$.
\end{lemma}

\begin{proof}
Consider $\gamma$ as a self-looped graph with $c$ vertices and $e$ edges. This yields a cellular decomposition of $\R^2 \cup \{\infty\}$ with, say $g+1$ faces, one of which is unbounded and contains the point at infinity. Hence, $(g+1) - e + c = \chi(S^2) = 2$. Likewise, $\gamma$ also gives a cellular decomposition of $S^1$ with $e$ edges and $2c$ vertices, as each crossing represents two vertices. Thus we also have $-e + 2c = \chi(S^1) = 0$. Rearranging this system of equations gives $g = c+1$.
\end{proof}

We can now prove Theorem \ref{knotflattening}.

\begin{proof}
Without loss of generality, we can assume $\gamma$ lies on the $xy$-plane. Let $\hat{\Phi}$ be the electric potential of $\gamma$, defined on $S^3 - \gamma$, and let $M_i$ denote the number of critical points of $\hat{\Phi}$ of index $i$. Since $\gamma$ is a plane curve, the electric field passes to a vector field on the plane, whose zeros are in bijective correspondence with the critical points of $\hat{\Phi}$. As the electric field points towards $\gamma$ near the curve, by the Poincare-Hopf Index Theorem, we can deduce each bounded planar component has at least one source or sink zero by taking a loop sufficiently close to the boundary of each said component. These zeros must be sources, because a sink would correspond to a local maximum of $\hat{\Phi}$, which cannot occur because $\hat{\Phi}$ is harmonic. These sources correspond to critical points of index $1$ in $S^3 - \gamma$, so applying Lemma \ref{eulerknot}, we see that $M_1 \geq c + 1$. 

Now take $\delta > 0$ sufficiently small such that by taking a $\delta$-lifting $K$ having the prescribed knot type with electric potential $\Phi$, the gradient vector field $\nabla \Phi$ is topologically equivalent to $\nabla \hat{\Phi}$ around sufficiently small neighborhoods each of the aforementioned critical points of index $1$. This is possible due to the dense structural stability of gradient vector fields, albeit we may have to apply an arbitrarily small perturbation to $K$ \cite{palis83}.

We will now use $m_i$ to denote the critical points of $\Phi$ with index $i$. By structural stability, we have that $m_1 \geq M_1$. Note that the $\delta$-lifting could introduce new critical points of index $1$ within the balls, which is why we have an inequality instead of strict equality. Hence, by Lemma \ref{m2m1}, we get $m_2 = m_1 - 1 \geq c$. Adding up the $m_i$, we get that the total critical set of $\Phi$ includes at least $2c + 2$ critical points.
\end{proof}
\begin{remark}
By repeatedly applying Reidemeister I moves to a given knot diagram, we can realize parametrizations with arbitrarily large critical sets for any given knot type.
\end{remark}
This partial sharpening does not improve the bound on $cp(K)$, as this would require us to realize a knot parametrization with precisely $2c+2$ critical points. Even in the potential of a planar curve, a connected planar component can have sources as well as saddles. For instance, the planar curve defined by $r(t)= ((2+1.5\cos{2t})\cos{t}, (2+1.5\cos{2t})\sin{t})$ has a saddle at the origin with one source in each of the two leaves.

\section{The relationship between the critical set and the Morse code of equipotential surfaces}
In this section, we define the notion of a ``Morse code" of an embedded knot. The Morse code lists the genera of the equipotential surfaces in the order with which they appear as one increases the potential from zero to infinity. Calculating the Morse code rigorously is difficult in general, but we can empirically observe the Morse code through computer visualization by applying the Marching Cubes algorithm on the potential \cite{knotcode}, as seen in Figure 2. 

Consider $\tx{Crit}(\Phi)$ and label the points $p_0, p_1, \dots, p_N$ such that $p_0 = \infty$, and $\Phi(p_i) \leq \Phi(p_j)$ whenever $i \leq j$, and let $V_0 = 0 < V_1 < \dots < V_{N'}$ be the distinct critical values of $\Phi$. 
\begin{definition}
 A critical set is \textbf{distinct} when each of the critical values $f(p_0), \dots, f(p_N)$ are distinct.
\end{definition}
Note that unless the critical set is distinct, we do not necessarily have that $\Phi(p_i) = V_i$ for all $i$. Let $G_i$ be the genus of the surface $S_i = \Phi^{-1}(V_i + \vep)$ where $\vep > 0$ is chosen such that $V_i + \vep_0$ is a regular value for each $0 < \vep_0 \leq \vep$. Because there are only finitely many $i$, a single $\vep$ can be chosen to work for all $i$. 

By the implicit function theorem, each $S_i$ is a smooth, compact, orientable surface without boundary embedded in $\R^3$. By the Morse Reconstruction Lemma \cite{nico} the topology of each $S_i$ is fixed for properly chosen $\vep$, and thus each $G_i$ is well-defined.
\begin{definition}
 The \textbf{Morse code} of $K$ is the $N'$-tuple $(G_0, G_1, \dots, G_{N'})$. A Morse code is \textbf{distinct} when the critical set is distinct.
\end{definition}
\begin{remark}
  Clearly, for a distinct critical set, $N = N'$. 
\end{remark}
\begin{figure}[h]
\centerline{\includegraphics[width=4.5in]{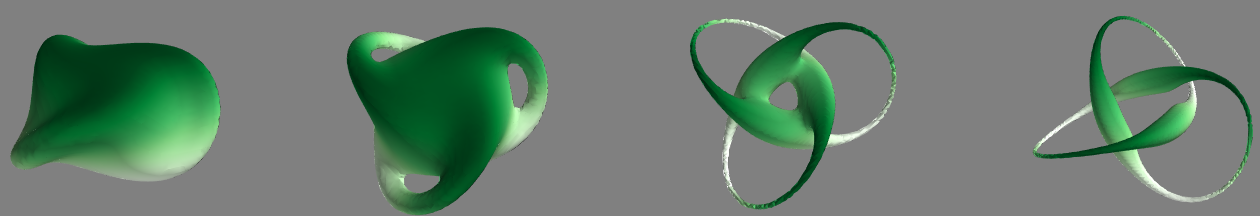}}
\caption{The four surfaces of the Morse code for a trefoil knot with order $3$ symmetry. We can read off the nondistinct Morse code as $(0,3,4,1).$}
\label{trefcode}
\end{figure}
\begin{remark} Not all critical sets and Morse codes are distinct, as can be seen in the example of the trefoil with order $3$ rotational symmetry. Distinctness is a generic property, however. By introducing a perturbation in the knot, we can change a nondistinct Morse code into a distinct Morse code.
\end{remark}.
\begin{lemma}
\label{codeends}
All Morse codes begin with $0$ and end with $1$.
\end{lemma}
\begin{proof}
The first statement follows from a multipole expansion sufficiently far away from the origin. The second statement follows from the fact that the electric field is approximately normal to the charge distribution in a sufficiently small neighborhood. See Chapters 2.5 and 3.4 of Griffiths \cite{griffiths}.
\end{proof}
Intuitively, this means the equipotential surface for a very small voltage level will be a large topological spheroid surrounding the knot, and the equipotential surface for a very high voltage level will be a thin tube surrounding the knot.

\begin{example} The most basic example of a Morse code comes from the trivial embedding of the unknot, defined by 
$r(t) = (\cos{t},\sin{t},0)$. We can immediately see that any finite critical point must lie on the $z=0$ plane, as otherwise the $z$-coordinate of the electric field will be the integral of a continuous strictly nonzero function divided by a strictly positive function, which is nonzero. We can then see that the origin is a critical point by symmetry, as the force exerted on the origin by any point on the knot is cancelled by an opposing force on the point reflected through the origin. It is a routine calculation to see that there are no other critical points on the plane, as seen in Zypman \cite{zypman}. 

This critical point has index $1$, as a perturbation in the $z=0$ plane will send a test charge towards the knot, whilst a perturbation along the $z$-axis will send a test charge back towards the origin, as all attracting charge is concentrated in the $z=0$ plane. As for the Morse code, there are only two topologically inequivalent equipotential surfaces, so by Lemma \ref{codeends}, the Morse code is $(0,1)$.
\end{example}

We now come to the main theorem of this section.

\begin{theorem}
\label{morsecode}
Suppose $(G_0, \dots, G_N)$ is the distinct Morse code of $K$.

\begin{enumerate}
    \item If $p_i$ has index $1$, then $G_i = G_{i-1} + 1$.
    \item If $p_i$ has index $2$, then $G_i = G_{i-1} - 1$.
\end{enumerate}
\end{theorem}

\begin{proof}
We start with the first statement. Suppose $p_i$ has index $1$. Let $M_i = \Phi^{-1}([0, \Phi(p_i) + \vep])$, which is a handlebody submanifold of $S^3$, and let $S_i = \partial M_i$, which as we saw, is a smooth, compact, orientable surface of genus $G_i$. By the Morse Reconstruction Lemma, $M_i$ is homotopy equivalent to $M_{i-1}$ with a one dimensional disc attached. Hence, by taking a tubular neighborhood in $S^3$, we can see it is homotopy equivalent to a handlebody homeomorphic to $M_{i-1}$ with another handle attached. Therefore, $G_i = G_{i-1} + 1$.

The second statement involves a somewhat roundabout argument. Suppose $p_i$ has index $2$. Then $M_i$ is homotopy equivalent to $M_{i-1}$ with a two dimensional disc attached. Let $M = \Phi^{-1}([0,V_i - \vep])$, where $\vep$ is sufficiently small such that $M$ is homeomorphic to $M_{i-1}$, which in turn implies $\partial M = S = \Phi^{-1}(V_i - \vep)$ is homeomorphic to $S_{i-1}$. Because $M_0 \subseteq M_1 \dots \subseteq M_N$, with the inclusions being strict by distinctness of the critical set, this disc is attached via an immersion $\gamma: S^1 \to S$, and we can say $\gamma \in \pi_1(S)$.

The fundamental group $\pi_1(S)$ is the free group on $2G_{i-1}$ generators modulo a product of $G_{i-1}$ commutators of paired generators. Observe that $\pi_1(S_i)$ is the quotient of $\pi_1(S)$ by a certain word $w$ which represents $\gamma$. Attaching a disc to a topological space quotients the original fundamental group by the attaching map, because the path around the disc is nullhomotopic in the new space. Hence, by the classification of compact orientable surfaces by their fundamental groups, $G_i \leq G_{i-1}$, since $\pi_1(S_i)$ does not have more generators than $\pi_1(S)$. We show this inequality is strict by proving $w$ is not the trivial word.

\begin{figure}[h]
\centerline{\includegraphics[width=3.5in]{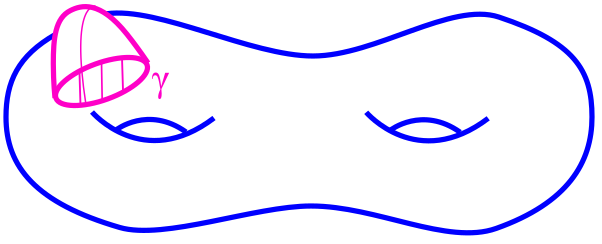}}
\caption{The attachment of a $2$-disc to $S_{i-1}$ (not necessarily of genus $2$) via a nullhomotopic loop. Note the attachment actually occurs in the interior of the solid surface, as the gradient points inwards in the equipotential surfaces.}
\end{figure}

Suppose $w$ is trivial. Then $\gamma$ bounds a disc on $S_{i-1}$. Note that the Morse Reconstruction Theorem states $\gamma$ does not self-intersect. By the construction of cell complexes, we also have that $\gamma$ bounds another embedded disc in $\R^3 - S_{i-1}$. See the above figure. The union of these two discs forms an embedded $2$-sphere $\tilde{S}$. By the Jordan-Brouwer separation theorem, we have that $S_{i-1} \cup \tilde{S}$ partitions $S^3$ into three disjoint connected components. 

By taking a tubular neighborhood and then taking the boundary, we can see that $S_i$ is a disjoint union of $S_{i-1}$ and a sphere $S^2$. The flux integral of the electric field over a connected component of an equipotential surface is strictly positive. Therefore, each of these equipotential surfaces encapsulates electric charge by Gauss' Law, but this contradicts the fact that our charge distribution, the knot $K$, is connected by assumption. Hence, $w$ is nontrivial and we can see that $G_i < G_{i-1}$.

Now we can the facts that $m_2 = m_1 - 1$ and that $m_1 + m_2 + 1 = N$ to show $G_i = G_{i-1} - 1$. There are $m_1$ instances of consecutive terms in the Morse code increasing by $1$ by the first statement of the theorem. Therefore, there are $m_1 - 1$ instances of consecutive terms of the Morse code decreasing by a strictly positive integer. Since the Morse code begins at $0$ and ends at $1$, we can deduce that each of these reductions must be a reduction by $1$.
\end{proof}

We now state a corollary on the structure of a Morse code. 

\begin{corollary}
The lowest nonzero critical value $V_1$ of a distinct Morse code cannot correspond to a critical point of index $2$. Therefore, $G_1 = 1$.
\end{corollary}

\begin{proof}
If $V_1$ corresponds to a critical point of index $2$, we would have that $\Phi^{-1}([0,V_1 + \vep])$ is homotopic to a $2-$sphere with a $2-$disc attached. But as we saw in the proof of Theorem \ref{morsecode}, we would get a contradiction because the attaching map would be nullhomotopic, as $S^2$ is simply connected.
\end{proof}

\begin{remark} The proof of the main result from Lipton \cite{lipton} uses the Morse Rearrangement Lemma, which shuffles around the critical values whilst leaving the critical set and its index data unchanged. This means one can replace $\Phi$ with another Morse function whose first nonzero critical value corresponds to a critical point of index $2$. However, this does not contradict the preceding corollary because the replacement is not necessarily harmonic. Hence, the flux integrals of the electric field over the replacement's level surfaces are not necessarily strictly positive.
\end{remark}

As stated earlier, many knot embeddings do not have distinct Morse codes. However, there is a straightforward generalization of Theorem \ref{morsecode}, whose proof we omit. This theorem relies on the fact that critical points of harmonic functions are isolated, applying the method of proof from Theorem \ref{morsecode} simultaneously across all critical points of the same potential.

\begin{theorem}
Now assume $K$ does not necessarily have a distinct Morse code. Suppose $p_{1}, \dots, p_{m}, q_{1}, \dots, q_{n}$ be all of the critical points of $\Phi$ with the same critical value $V$, such that each $p_i$ has index $1$ and each $q_j$ has index $2$. If the preceding equipotential surface has genus $G$, then the successive equipotential surface has genus $G + m - n$. Therefore, the Morse code is of the form $(0, \dots, G, G + m - n, \dots, 1)$. We also have that when $V = V_1$, $m > n$.
\end{theorem}
\section{Concluding Remarks and Future Directions}
There are still unresolved questions in the potential theory of knot complements, and the ramifications to broader research on knot energies and flows have yet to be explored, particularly the bifurcation of critical points induced from these specific flows. Whilst degenerate critical points have little physical significance, the collapse of the critical set to a size strictly below $cp(K)$ would likely yield new insights into the symmetries of certain knot classes.

Experimental mathematics can prove to be a potent tool to advance the theory as well, should some issues with computational expense be resolved. Figures \ref{unknotcritset} and \ref{trefcode} were able to be generated with minimal difficulty, but producing similar images for more complex knots will require significant computational resources. The author was able to reduce computation time by significant margins by implementing multicore parallel processing to divide the multivariable Newton method into smaller steps. However, there are further optimizations that have yet to be implemented, such as the use of the Fast Multipole Method in the evaluation of the potential integrals. Should these issues be resolved, we could produce experimental data for a larger class of knots, and formulate conjectures for some of the open questions in electrostatic knot theory.

\section{Acknowledgements}

The author would like to thank his advisor, Steven Strogatz, for introducing the problem to the author and for his helpful comments on this article. He would also like to thank Alex Townsend for his assistance with numerics and visualization. Finally, the author would like to thank the staff of Cornell's IT Department for their technical assistance.

\printbibliography

\end{document}